\documentclass[a4paper,12pt,oneside]{amsart} 
\usepackage[ansinew]{inputenc}
\usepackage[british]{babel} 
\usepackage{amsmath}
\usepackage{amssymb}
\usepackage{amsfonts}
\usepackage{latexsym}
\usepackage{amsthm}
\usepackage{enumerate}
\theoremstyle{plain}
\newtheorem{theorem}{Theorem}[section]
\newtheorem{lemma}[theorem]{Lemma}
\newtheorem{corollary}[theorem]{Corollary}
\theoremstyle{definition}
\newtheorem{definition}[theorem]{Definition}
\newtheorem{example}[theorem]{Example}
\theoremstyle{remark}
\newtheorem*{remark}{Remark}
\newcommand{\R}{\mathbb{R}}
\newcommand{\N}{\mathbb{N}}
\newcommand{\Rn}{\mathbb{R}^{n}}

\newcommand{\BV}{\mathrm{BV}}
\newcommand{\Lp}{L^{p}}
\newcommand{\Np}{N^{1,p}}
\newcommand{\No}{N^{1,1}}
\newcommand{\liploc}{\mathrm{Lip}_{\mathrm{loc}}}
\newcommand{\Nlocp}{N^{1,p}_{\mathrm{loc}}}
\newcommand{\Nloco}{N^{1,1}_{\mathrm{loc}}}
\newcommand{\ep}{\varepsilon}
\newcommand{\lloc}{L^{1}_{\mathrm{loc}}}

\newcommand{\modp}{\mathrm{Mod}_{p}}
\newcommand{\modo}{\mathrm{Mod}_{1}}
\newcommand{\capp}{\mathrm{cap}_{p}}
\newcommand{\capo}{\mathrm{cap}_{1}}
\newcommand{\loclipcapp}{\mathrm{locLip}-\mathrm{cap}_{p}}
\newcommand{\loclipcapo}{\mathrm{locLip}-\mathrm{cap}_{1}}
\newcommand{\newtcapp}{N^{1,p}_{\mathrm{loc}}-\mathrm{cap}_{p}}
\newcommand{\newtcapo}{N^{1,1}_{\mathrm{loc}}-\mathrm{cap}_{1}}
\newcommand{\capbv}{\mathrm{cap}_{\mathrm{BV}}}
\newcommand{\cbv}{\widetilde{\mathrm{cap}}_{\mathrm{BV}}}
\newcommand{\cbvo}{{\mathrm{cap}}_{\mathrm{BV,O}}}
\newcommand{\capone}{{\mathrm{cap}}_{1,\mathrm{O}}}

\newcommand{\cbvtr}{\widetilde{\mathrm{cap}}_{\mathrm{BV-trunc}}}

\newcommand{\haus}{\mathcal{H}}
\newcommand{\intave}{\dashint}

\DeclareMathOperator*{\aplim}{\text{ap lim}}
\def\Xint#1{\mathchoice
{\XXint\displaystyle\textstyle{#1}}%
{\XXint\textstyle\scriptstyle{#1}}%
{\XXint\scriptstyle\scriptscriptstyle{#1}}%
{\XXint\scriptscriptstyle\scriptscriptstyle{#1}}%
\!\int}
\def\XXint#1#2#3{{\setbox0=\hbox{$#1{#2#3}{\int}$}
\vcenter{\hbox{$#2#3$}}\kern-.5\wd0}}

\def\dashint{\Xint-}
\begin{document}
\title[On BV-capacities and Sobolev capacity]{Comparisons of relative BV-capacities and Sobolev capacity in metric spaces}
\author[Hakkarainen \and Shanmugalingam]{Heikki Hakkarainen \and Nageswari Shanmugalingam}

\begin{abstract}
We study relations between the variational Sobolev $1$-capacity and versions of variational BV-capacity 
in a complete metric space equipped with a doubling measure and supporting a weak $(1,1)$-Poincar\'e inequality.
We prove the equality of $1$-modulus and $1$-capacity, 
extending the known results for $1<p<\infty$ to also cover the more geometric case $p=1$. 
Then we give alternative definitions for variational BV-capacities and obtain equivalence results between them. 
Finally we study relations between total $1$-capacity and versions of BV-capacity.
\end{abstract}

\maketitle
\section{Introduction}
In this article we study connections between different $1$-capacities and BV-capacities in 
the setting of metric measure spaces. We obtain characterizations for the $1$-capacity of a condenser, 
extending  results known for $p>1$ to also 
cover the more geometric case $p=1$. Furthermore, we study how the different versions of BV-capacity, corresponding to 
various pointwise requirements that the capacity test functions need to fulfill, relate to each other. 

One difficulty that occurs when working with minimization problems on the space $W^{1,1}(\R^n)$ is the lack of reflexivity. 
Indeed, the methods used to develop the theory of $p$-capacity are closely related to those used in certain variational minimization problems; in such problems reflexivity
or the weak compactness property of the function space $W^{1,p}(\R^n)$ 
when $p>1$ usually plays an important role. 
One possible way to deal with this lack of reflexivity is to consider the space $\BV(\R^n)$, that is, the
class of functions of bounded variation. 
This wider class of functions provides tools, such as lower
semicontinuity of the total variation measure, that can be used to overcome the problems 
caused by the lack of reflexivity in the arguments. 

This approach was originally used to study variational $1$-capacity in the Euclidean case in~\cite{Pi73}, \cite{Hut},
and~\cite{FedZ72}. The article 
\cite{KinKST08} showed that similar approach could also be used to study variational $1$-capacity in the setting of metric measure spaces. In \cite{KinKST08} the main tool in obtaining a connection between the 
$1$-capacity and BV-capacity was the metric space version of Gustin's boxing inequality, see \cite{Gus60}. Since then, this 
strategy has been used in \cite{HakKin} to study a version of BV-capacity and Sobolev $1$-capacity in the setting of 
metric measure spaces. Since the case $p=1$ corresponds to geometric objects in the metric measure space such as sets of finite
perimeter and minimal surfaces, it behooves us to understand this case. A study of potential theory in the case $p=1$
for the setting of metric measure spaces
was begun in the papers~\cite{HakKin}, \cite{KinKST08}, and~\cite{KinKST10}, and we continue the study in this note.

In the theory of calculus of variations, partial differential equations and potential theory many interesting features of Sobolev functions are measured in terms of capacity.
Just as sets of measure zero are associated with 
the $L^p$-theory, sets of $p$-capacity zero should be understood in order to deal with the Sobolev space $W^{1,p}(\Rn)$.
For $1<p<\infty$ the $p$-capacity of sets is quite well-understood; see for example \cite{EvaG92}, \cite{Maz}, 
\cite{Zie89} for the Euclidean setting, 
and in the more general metric measure space setting, \cite{HeiK98}, \cite{KalS01}, \cite{GolT02}, \cite{KinMa1}, \cite{KinMa2},
\cite{KinMa3}, \cite{Cos09}, and the references therein. 
The situation corresponding to the case $p=1$
is not so well-understood. 

In the setting of $p=1$ there are two natural function spaces to consider: Sobolev type
spaces, and spaces of functions of bounded variation. Since these two spaces are interrelated, it is natural
to expect that the corresponding capacities are related as well. However, these two function spaces are fundamentally
different in nature. Functions in the Sobolev class $W^{1,1}(\R^n)$ have quasicontinuous representatives, 
whereas functions of bounded
variation need not have quasicontinuous representatives, and in fact sometimes exhibit jumps across sets of nonzero
codimension one Hausdorff measure;  see for example~\cite{AmbMP04}. 
Because of the quasicontinuity of Sobolev type functions, one can either insist on the test functions
to have value one in a neighborhood of the set whose capacity is being computed, or merely require the test functions to 
have value $1$ on the set.
Such flexibility is not available in computing the BV-capacities and hence different point-wise requirements on the test 
functions might lead to different types of BV-capacity.
Thus, corresponding to the BV-class there is more than one possible notion of capacity and
it is nontrivial to even know which versions of BV-capacities are equivalent.  We point out here that the 
analog of Sobolev spaces considered in this paper, called the Newton-Sobolev spaces, consist automatically only
of quasicontinuous functions when the measure on $X$ is doubling and the space supports a $(1,1)$-Poincar\'e inequality.


The primary goal of this note is to compare these different notions of capacity related to the BV-class, and
to the capacity related to the 
Newton-Sobolev space $N^{1,1}(X)$. In Section~2 we describe the objects considered in this paper, and then in
Section~3 we show that the $1$-modulus of the family of curves in a domain $\Omega$ in the metric space, connecting 
two nonempty pairwise disjoint 
compact sets $E,F\subset\Omega$, is the same as the Newton--Sobolev $1$-capacity and local Lipschitz capacity of the condenser $(E,F,\Omega)$. Our arguments are based on constructing
appropriate test functions. In Section~4 we consider three alternative notions of variational BV-capacity of a compact set
$K\subset \Omega$, and show that these notions are comparable to each other and to the variational
Sobolev $1$-capacity of $K$ relative to $\Omega$. We combine tools such as discrete convolution and boxing 
inequality to obtain the desired results. 
In Section~5 we consider total capacities where we minimize the norm of the test functions rather than their energy 
seminorm. We consider these quantities
of more general bounded sets - related to both BV-functions and to Sobolev functions - and compare them to their variational capacities. 

In the classical
Euclidean setting some of these results are known, for example from \cite{Hut}, but even in the weighted Euclidean setting and the Carnot groups the results of this paper are new.

{\bf Acknowldegement:} N.S. was partially supported by the Taft Foundation of the University of Cincinnati. H.H.~was 
supported by The Finnish National Graduate School in Mathematics and its Applications. 
Part of the research was conducted during H.H.'s visit to the University of Cincinnati and N.S.'s visit to Aalto University; the
authors thank these institutions for their kind hospitality.
%
%
%
%
%
%
\section{Preliminaries}
In this article $X=\left(X,d,\mu \right)$ is a complete metric measure space. We assume that $\mu$ is a Borel regular outer measure that is \emph{doubling}, i.e. there is a constant $C_D$, called the \emph{doubling constant} of $\mu$, such that
\[
\mu(2B)\leq C_D\mu(B)
\]
for all balls $B$ of $X$. Furthermore, it is assumed that $0<\mu(B)<\infty$ for all balls $B\subset X$. 
These assumptions imply that the metric space $X$ is \emph{proper}, that is, 
closed and bounded sets are compact. In this article, by a \emph{path} we mean a rectifiable nonconstant continuous
mapping from a compact interval to $X$. 
The standard tool that is used to measure path families is the following.
\begin{definition}
Let $\Omega\subset X$ be an open connected set. The \emph{$p$-modulus} of
a collection of paths $\Gamma$ in $\Omega$ for $1\leq p<\infty$ is defined as
\[\modp(\Gamma) = \inf\limits_{\varrho} \int_{\Omega} \varrho^{p} d\mu \] 
where the infimum is taken over all nonnegative Borel-measurable functions $\varrho$ such that 
$\int_{\gamma} \varrho\, ds \geq 1$ for each $\gamma\in\Gamma$.
%
\end{definition}
We use the definition of Sobolev spaces on metric measure space $X$ based on the notion of 
$p$-weak upper gradients, see~\cite{HeiK98}, \cite{Sha00}. Note that instead of the whole space $X$, we can consider its open subsets in the definitions below.  
\begin{definition}
A nonnegative Borel function $g$ on $X$ is an \emph{upper gradient} of an extended real valued
function $u$ on $X$ if for all paths $\gamma$, 
\begin{equation}\label{uppergrad}
|u(x)-u(y)|\leq\int_{\gamma}g\,ds,
\end{equation}
whenever both $u(x)$ and $u(y)$ are finite, and $\int_{\gamma}g\,ds=\infty$ otherwise. Here
$x$ and $y$ denote end points of $\gamma$.
Let $1\le p<\infty$. If $g$ is a nonnegative measurable function on $X$, and if the integral in 
\eqref{uppergrad} is well defined and the inequality holds for $p$-almost every path, then $g$ is a 
\emph{$p$-weak upper gradient} of $u$.
\end{definition}
The phrase that inequality \eqref{uppergrad} holds for \emph{p-almost every path} with $1\leq p <\infty$ means that it fails
only for a path family with zero $p$-modulus. 
Many usual rules of calculus are valid for upper gradients as well. For a good overall reference interested readers may
see~\cite{Hei01},  \cite{Haj03}, and~\cite{BjoB10}. 
%

If $u$ has a $p$-weak upper gradient $g\in L^p_{\text{loc}}(X)$, then
there is a \emph{minimal $p$-weak upper gradient} $g_u$ such that $g_u\le g$ $\mu$-almost everywhere 
for every $p$-weak upper gradient $g$ of $u$, see \cite{Haj03} and \cite{BjoB10}. 

%
When $\Omega$ is an open subset of Euclidean space equipped with the Euclidean metric and Lebesgue measure, the Sobolev type
space considered below coincides with the space of quasicontinuous representatives of classical Sobolev functions.
In the setting of weighted Euclidean spaces with $p$-admissible weights, or the Carnot-Carath\'eodory spaces, the
corresponding Sobolev spaces coincide in an analogous manner with the Sobolev type space considered below. 
We refer the interested reader to~\cite[Example~3.10]{Sha00}; recall that in the Carnot-Carath\'eodory setting, the classical Sobolev
spaces consist of functions in $L^p(X)$ whose \emph{horizontal} derivatives are also in $L^p(X)$, and so the 
upper gradient analog is with respect to paths with tangents in the horizontal directions. In the Carnot-Carath\'eodory
setting paths with finite lengths are precisely paths with tangents in the horizontal directions.

\begin{definition}
Let $1\leq p<\infty$. If $u\in\Lp(X)$, let
\[
\Vert u\Vert_{N^{1,p}(X)}=\left(\int_{X}|u|^{p}\,d\mu+ \inf\limits_{g}\int_{X} g^{\,p}\,d\mu\right)^{1/p},
\]
where the infimum is taken over all $p$-weak upper gradients of $u$. The \emph{Newtonian space} on $X$ is the
quotient space
\[
N^{1,p}(X)=\left\{u: \Vert u\Vert_{N^{1,p}(X)}<\infty\right\}\big /\sim,
\]
where $u\sim v$ if and only if $\Vert u-v\Vert_{N^{1,p}(X)}=0$.
\end{definition}
Note that if $X$ has no nonconstant rectifiable path, then zero is an upper gradient of every function, and so
in this case $N^{1,p}(X)=\Lp(X)$. Therefore, to have a reasonable theory of Sobolev spaces that gives rise to a viable
potential theory, one needs a further condition on the metric measure space. In literature, the following Poincar\'e type inequality is considered.
\begin{definition}
Let $1\le p<\infty$.
The space $X$ supports a \emph{weak (1,p)-Poincaré inequality} if there exists constants $C_{P}>0$ and 
$\tau\geq 1$ such that for all balls $B(x,r)$ of $X$, all locally integrable functions $u$ on $X$, and for all $p$-weak upper gradients $g$ of $u$,
\[\intave_{B(x,r)}|u-u_{B(x,r)}|\,d\mu\leq C_{P}r\left(\intave_{B(x,\tau r)}g^{p}\,d\mu\right)^{1/p},\]
where
\[u_{B(x,r)}=\intave_{B(x,r)}u\,d\mu=\frac{1}{\mu(B(x,r))}\int_{B(x,r)}u\,d\mu.\]
\end{definition}

It is known that $\text{Lip}(X)\cap N^{1,p}(X)$ is dense in $N^{1,p}(X)$ if $\mu$ is doubling and $(1,p)$-Poincaré inequality is satisfied, see \cite{Sha00}. From this it easily follows that Lipschitz functions with compact support are dense in $N^{1,p}(X)$, if $X$ is also complete. Furthermore, under a Poincar\'e inequality it is known that the metric space supports a myriad of rectifiable
curves; see for example~\cite{HeiK98} and~\cite{Kor07}.

The theory of functions of bounded variation and sets of finite perimeter in metric measure space setting will
be extensively used in this thesis. These concepts were introduced and developed in~\cite{Mir03}, \cite{Amb01}, \cite{Amb02}, \cite{Amb03},\cite{AmbM03} and \cite{AmbMP04}.

\begin{definition}
Let $\Omega\subset X$ be an open set. The \emph{total variation} of a function $u\in L^1_{\text{loc}}(\Omega)$ is defined as
\[
    \Vert Du\Vert(\Omega)=\inf_{\{u_i\}}\liminf_{i\to\infty}\int_{\Omega} g_{u_i}\, d\mu,
\]
where $g_{u_i}$ is an upper gradient of $u_i$ and the infimum is taken over all such sequences of functions 
$u_i\in \liploc(\Omega)$ such that $u_i\to u$ in 
$\lloc(\Omega)$.
We say that a function $u\in L^{1}(\Omega)$ is of \emph{bounded variation}, denoted $u\in\BV(\Omega)$, 
if $\Vert Du\Vert(\Omega)<\infty$. Furthermore, we say that 
$u$ belongs to $\BV_{\text{loc}}(\Omega)$ if $u\in\BV(\Omega')$ for every $\Omega'\Subset \Omega$. Note that 
$N^{1,1}(\Omega)\subset \BV(\Omega)$.

It was shown in~\cite{Mir03} that for $u\in\BV_{\text{loc}}(X)$ we have
$\Vert Du\Vert(\cdot)$ to be a Radon measure on $X$. If $U\subset X$ is an open set such that $u$ is constant on $U$, then
$\Vert Du\Vert(U)=0$. A Borel set $E\subset X$ is said to have \emph{finite perimeter}
if $\chi_E\in \BV(X)$; the perimeter $P(E,A)$ of $E$ in a Borel set $A\subset X$ is the number 
\[
  P(E,A)=\Vert D\chi_E\Vert(A).
\]
\end{definition}
\begin{remark}
Since $X$ is a complete and doubling $1$-Poincaré space, we can consider sequences from $N^{1,1}_{\text{loc}}(\Omega)$ instead of from $\liploc(\Omega)$ in the definition of $\Vert Du\Vert(\Omega)$. 
This is due to the fact that functions from $N^{1,1}_{\text{loc}}(\Omega)$ can be approximated by functions from $\liploc(\Omega)$.
For a discussion, see for instance~\cite[Theorem~5.9]{BjoBS08} and~\cite[Theorem~5.41]{BjoB10}.
\end{remark}

The next lemma is useful in the following sections of this paper.

\begin{lemma}\label{compact}
 Let $\Omega$ be an open subset of $X$ and $u\in \BV(\Omega)$ such that the support of $u$ is a compact subset of $\Omega$. 
 Then there is an open set $U\Subset\Omega$ with $\text{supt}(u)\Subset U$, and a sequence $u_i\in \mathrm{Lip}(\Omega)$
 with $\text{supt}(u_i)\subset U$, such that $u_i\to u$ in $L^1(\Omega)$ and 
 \[
    \Vert Du\Vert(\Omega)=\lim_{i\to\infty}\int_{\Omega} g_{u_i}\, d\mu
 \]
 for a choice of upper gradient $g_{u_i}$ of $u_i$.
\end{lemma}

\begin{proof}
Since $u\in \BV(\Omega)$ there is a sequence $v_i\in \liploc(\Omega)$ with $v_i\to u$ in $L^1_{\text{loc}}(\Omega)$ and
\[\Vert Du\Vert(\Omega)=\lim_{i\to\infty}\int_{\Omega} g_{v_i}\, d\mu\] 
for some choice of upper gradient $g_{v_i}$ of $v_i$.
Since $u$ has compact support in $\Omega$ we can find an open set $U\Subset\Omega$ with $\text{supt}(u)\Subset U$.
Let $\eta$ be an $L$-Lipschitz function on $\Omega$ such that $0\le \eta\le 1$, $\eta=1$ on $\text{supt}(u)$, and $\eta=0$
on $\Omega\setminus U$. We set $u_i=\eta\, v_i$, and now show that this choice satisfies the claim of the lemma.

To see this, note that
\begin{align*}
  \int_{\Omega}|u-u_i|\, d\mu&=\int_U|u-\eta\, v_i|\, d\mu\\ 
      &=\int_{\text{supt}(u)}|u-v_i|\, d\mu + \int_{U\setminus\text{supt}(u)}\eta|v_i|\, d\mu\\
      &\le \int_U|u-v_i|\, d\mu+\int_{U\setminus\text{supt}(u)}|u-v_i|\, d\mu\\
      &\le 2\int_U|u-v_i|\, d\mu\to 0\; \text{ as }i\to\infty.
\end{align*}
Furthermore, the function $g_{u_i}=\eta g_{v_i}+|v_i|\, g_\eta$ is an upper gradient of $u_i$. Because $g_\eta=0$
on $\text{supt}(u)\cup (\Omega\setminus U)$ and $g_\eta\le L$, we see that
\begin{align*}
  \int_{\Omega}g_{u_i}\, d\mu&\le \int_Ug_{v_i}\, d\mu+L\int_{U\setminus\text{supt}(u)}|v_i|\, d\mu\\
       &\le \int_{\Omega}g_{v_i}\, d\mu+L\, \int_{U\setminus\text{supt}(u)}|u-v_i|\, d\mu.
\end{align*}
Because 
\[\lim_{i\to\infty}\int_{U\setminus\text{supt}(u)}|u-v_i|\, d\mu=0,\] 
it follows that
\[
   \limsup_{i\to\infty}\int_{\Omega}g_{u_i}\, d\mu\le \Vert Du\Vert(\Omega).
\]
On the other hand, by the previous argument we know that $u_i\to u$ in $L^1(\Omega)$. Thus 
\[\Vert Du\Vert(\Omega)\le \liminf_{i\to\infty}\int_{\Omega}g_{u_i}\, d\mu,\] 
and this completes the proof.
\end{proof}


The following \emph{coarea formula} holds, see~\cite[Proposition 4.2]{Mir03}: whenever $u\in \BV(X)$ and $E\subset X$ is a Borel set,
\[ 
   \Vert Du\Vert(E)=\int_{\R} P(\{u>t\},E)\, dt.
\]
It was also shown in the thesis~\cite[Theorem 6.2.2]{Cam08} that if $u\in N^{1,1}(X)$ then 
\[
    \Vert u\Vert_{N^{1,1}(X)}=\int_X|u|\, d\mu+\Vert Du\Vert(X)
\]
and that there is a minimal weak upper gradient $g$ of $u$ such that whenever $E\subset X$ is measurable, we have
\[\Vert Du\Vert(E)=\int_E g\, d\mu.\]
The above was proved in \cite[Theorem 6.2.2]{Cam08} for the case that $E$ is an open set. Since $\Vert Du\Vert(\cdot)$ is a Radon measure, 
the above follows for all measurable sets.
In a metric measure space equipped with a doubling measure, the $(1,1)$-Poincar\'e inequality implies 
the following \emph{relative isoperimetric inequality}, see~\cite{Mir03} or~\cite{AmbMP04}: 
there are constants $C>0$ and $\lambda\ge 1$ such that
whenever $E\subset X$ is a Borel set and $B$ is a ball of radius $r$ in $X$,
\[
   \min\{\mu(E\cap B),\mu(B\setminus E)\}\le C\, r\, P(E,\lambda B).
\]
See also~\cite[Theorem 1.1]{BobH97} for the converse statement. Thus if $E$ has zero perimeter in $\lambda B$, then the above inequality states that up to sets of measure zero, either $B\subset E$ or $B\cap E$ is empty. 

%
The \emph{Hausdorff measure of codimension one} of $E\subset X$ is defined as
\[
   \mathcal{H}(E)=\lim_{\delta\to 0}\ \inf\bigg\lbrace \sum_{i\in I} \frac{\mu(B_i)}{r_i}\, :\, I\subset\N,
                                       r_i\le \delta\text{ for all }i\in I, 
                                       \, E\subset\bigcup_{i\in I}B_i\bigg\rbrace.
\]  
Here $B_i$ are balls in $X$ and $r_i=\text{rad}~B_i$. 

\emph{Throughout this note we will assume that $(X,d,\mu)$ is a complete metric space with a doubling measure supporting a
weak $(1,1)$-Poincar\'e inequality.}
%
%
%
%
%
%
\section{Modulus and Continuous capacity of condensers; the case $p=1$}
It was shown in~\cite{KalS01} that when $1<p<\infty$ the $p$-modulus of the family of all paths in a domain $\Omega\subset X$
that connect two disjoint continua $E,F\subset \Omega$ must equal the variational continuous $p$-capacity of the condenser
$(E,F,\Omega)$. The tools used in that paper include a strong form of Mazur's lemma, which is not available when $p=1$. 
On the other hand, the results from~\cite{JarJRRS07}, which were not available at the time~\cite{KalS01} was written, should
in principle help us to overcome this difficulty. 

In this section we give an extension of the 
result of~\cite{KalS01} to the case $p=1$ by using the results of~\cite{JarJRRS07}. Once we have that the $p$-modulus of
the family of curves connecting $E$ to $F$ in $\Omega$ equals the variational $p$-capacity of the condenser
$(E,F,\Omega)$, we use the fact that under the Poincar\'e inequality Lipschitz functions form a dense subclass of the
space $N^{1,1}(X)$ and that functions in $N^{1,1}(X)$ are quasicontinuous to prove that the capacity of the condenser
can be computed by restricting to functions in $N^{1,1}(\Omega)$ that are continuous on $\Omega$. An easy 
modification of this proof also indicates that in computing the variational Sobolev $1$-capacity of a set $E\subset \Omega$,
one can insist that the test functions take on the value one on $E$, or equivalently take on the value one in a 
neighborhood of $E$, see for instance \cite{KilKM00} and \cite{BjoBS08}.

Let $\Omega\subset X$ be an open connected set in $X$, and
$E$ and $F$ be disjoint nonempty compact subsets of $\Omega$. Denote by $\modp(E,F,\Omega)$ the 
$p$-modulus of the collection of all rectifiable paths $\gamma$ in $\Omega$ with
one endpoint in $E$ and other endpoint in $F$. The \emph{$p$-capacity} of  the \emph{condenser} $(E,F,\Omega)$ is 
defined by
\[ 
    \capp(E,F,\Omega) = \inf \int_{\Omega} g^{p} d\mu , 
\]
where the infimum is taken over all nonnegative Borel measurable functions $g$ that are upper gradients of 
some function $u$, $0\leq u\leq 1$ with the property that $u=1$ in $E$ and $u=0$ in $F$. 
Note that in this definition it is not even required of $u$ to be measurable.
If $u$ is in addition assumed to be locally Lipschitz, then the corresponding number obtained
is denoted $\loclipcapp(E,F,\Omega)$. The definitions immediately imply that for $1\le p<\infty$,
%
\[\modp(E,F,\Omega) \leq \capp(E,F,\Omega) \leq \loclipcapp(E,F,\Omega). \]
By proposition 2.17 in \cite{HeiK98}, for $1\le p<\infty$ we have 
\[ \modp(E,F,\Omega) = \capp(E,F,\Omega).  \]
Our goal in this section is to prove that
\[\capo(E,F,\Omega) = \loclipcapo(E,F,\Omega) \]
and therefore extend the result in \cite{KalS01} to also cover the case $p=1$. We follow the strategy used in \cite{KalS01}. 
A crucial step in \cite{KalS01} is to prove that
\[\capp(E,F,\Omega) = \newtcapp(E,F,\Omega) , \]
where in the latter case the test functions $u$ are required to be in $\Nlocp(\Omega)$.
We prove that the above equality also holds for $p=1$. First we recall the following theorem. 
%
%
\begin{theorem}[{\cite[Theorem~1.11]{JarJRRS07}}]\label{measthm}
Let $X$ be a complete metric space that 
supports a doubling Borel measure $\mu$ which is nontrivial and finite on balls. Assume 
that $X$ supports a weak $(1,p)$-Poincar\'e inequality
for some $1\leq p<\infty$. If an extended real valued function $u:X\to\left[-\infty,\infty\right]$ 
has a $p$-integrable upper gradient, then $u$ is measurable and locally $p$-integrable.
\end{theorem} 

The following observation is an easy consequence of the previous theorem.
\begin{corollary}\label{measlem}
Let $\Omega\subset X$ be an open connected set. If $u:\Omega\to\left[0,1\right]$ has an 
upper gradient $g\in L^{p}(\Omega)$, then $u$ is measurable. 
\end{corollary}

\begin{proof}
Recall that we have as a standing assumption that $X$ supports a weak $(1,1)$-Poincar\'e inequality.

It suffices to prove that $u$ coincides with a measurable function locally. 
To do so, for $x_0\in\Omega$
let $r>0$ such that $B(x_0,4r)\subset \Omega$, and let $\eta$ be a $1/r$-Lipschitz function on $X$
such that $0\le \eta\le 1$, $\eta=1$ on $B(x_0,r)$, and $\eta=0$ on $X\setminus B(x_0,2r)$. Consider
the function $v=\eta u$; note that while $u$ is not defined on $X\setminus \Omega$, since $\eta=0$
on $X\setminus\Omega$ we can extend $v$ by zero to $X\setminus\Omega$. 

A direct computation shows that $\rho=\eta g+r^{-1}u\, \chi_{B(x_0,2r)\setminus B(x_0,r)}$ is 
an upper gradient of $v$ in $X$. Thus it follows from Theorem~\ref{measthm} that $v$ is 
measurable on $X$, and hence $v$ is measurable on the ball $B(x_0,r)\subset\Omega$.; the proof is
completed by noting that on $B(x_0,r)$ we have $u=v$.
\end{proof}

\begin{theorem}\label{onecapsobocap}
Let $E$ and $F$ be nonempty disjoint compact subsets of an open 
connected set $\Omega\subset X$. Then
\[\capo(E,F,\Omega) = \newtcapo(E,F,\Omega) . \]
\end{theorem}

\begin{proof}
It is clear that
\[\capo(E,F,\Omega) \leq \newtcapo(E,F,\Omega) . \]
%
Let $\ep>0$ and $u:\Omega\to\left[0,1\right]$ be a function such that $u=1$ in $E$ and 
$u=0$ in $F$ and with an upper gradient $g_{u}$ for which
\[\int_{\Omega} g_{u} d\mu < \capo(E,F,\Omega) + \ep .\]
By Corollary~\ref{measlem}, $u$ is a measurable function with an 
integrable upper gradient. Since $u$ is bounded it follows that $u\in\Nloco(\Omega)$. Hence
\[\newtcapo(E,F,\Omega) \leq \int_{\Omega} g_{u} d\mu < \capo(E,F,\Omega) + \ep.\]
The desired inequality follows by letting $\ep\to 0$. 
\end{proof}
As pointed out in the beginning of Section~2, a complete metric measure space $X$ with a doubling measure is \emph{proper}, i.e.~all 
closed and bounded subsets of $X$ are compact. The doubling condition of the measure 
and weak $(1,p)$-Poincar\'e inequality, with $1\le p<\infty$, imply also that
Lipschitz functions are dense in $\Np(X)$, see for example~\cite[Theorem 4.1]{Sha00}. For the following theorem we refer to~\cite[Theorem 1.1]{BjoBS08}.
\begin{theorem}\label{quasicont}
Let $X$ be proper, $\Omega\subset X$ open, $1\leq p<\infty$ and assume that continuous functions are dense 
in $\Np(X)$. Then every $u\in \Nlocp(\Omega)$ is quasicontinuous in $\Omega$; that is, for every $\ep>0$ there is
an open set $U_\ep$ with $\mathrm{Cap}_p(U_\ep)<\ep$ such that $u\vert_{\Omega\setminus U_\ep}$ is continuous on
$\Omega\setminus U_\ep$.
\end{theorem}

Here, by the total capacity $\text{Cap}_p(U_\ep)$ we mean the number
\[
   \text{Cap}_p(U_\ep)=\inf\Vert u\Vert_{N^{1,p}(X)}^p,
\]
where the infimum is taken over all $u\in N^{1,p}(X)$ with $u=1$ on $U_\ep$ and $0\le u\le 1$.

The following lemma (for $1<p<\infty$) is from \cite{KalS01}; we will prove that it also holds 
true in the case $p=1$. A similar proof can be used
to show that whenever $E\subset X$, we have $\text{Cap}_p(E)=\text{Cap}_{p,O}(E)$, where the latter is obtained by
considering test functions $u$ as in the definition of $\text{Cap}_p(E)$ but with the additional condition that $u\ge 1$ in a
neighborhood of $E$. See \cite{KilKM00} and \cite{BjoBS08} for further discussion.
\begin{lemma}\label{sobcapineq}
Let $E$ and $F$ be nonempty disjoint compact subsets of $\Omega$. For every 
$0<\ep<\frac{1}{2}$ there exists disjoint compact sets $E_{\ep}$ and $F_{\ep}$ of $\Omega$ such that for some
$\delta>0$
\begin{enumerate}[(i)]
\item $\overline{\bigcup_{x\in E}B(x,\delta)} = E_{\ep} $, \\
\item $\overline{\bigcup_{x\in F}B(x,\delta)} = F_{\ep} $, \\
\item $\newtcapo(E_{\ep},F_{\ep},\Omega) \leq \frac{1}{1-2\ep}\left(\newtcapo(E,F,\Omega)+2\ep\right)$.
\end{enumerate}
\end{lemma}

\begin{proof}
Let $0<\ep<1/2$ 
and $u\in\Nloco(\Omega)$ be such that 
$0\leq u\leq 1$, $u=1$ in $E$, $u=0$ in $F$, and with an upper gradient $g_{u}$ satisfying
\[ \int_{\Omega} g_{u} d\mu < \newtcapo(E,F,\Omega) + \ep . \]
Let  $L=4\,\text{dist}(E,F)^{-1}$. By Theorem \ref{quasicont} the function $u$ is quasicontinuous in $\Omega$. Hence there 
exists an 
open set $U$ such that $u\vert_{\Omega\setminus U}$ is continuous in $\Omega\setminus U$ and 
$\text{Cap}_{1}(U)<\ep/(1+L)$. 
Thus there exists  $v\in\No(X)$, $0\leq v\leq 1$,
such that $v=1$ in $U$ and
\[ \int_{X} v \, d\mu + \int_{X} g_{v} d\mu < \frac{\ep}{1+L} . \]  
Let $E_1=E\setminus U$ and $F_1=F\setminus U$; then $E_1, F_1$ are compact subsets of $\Omega\setminus U$ with
$E\setminus E_1=E\cap U$ and $F\setminus F_1=F\cap U$ subsets of $U$. Let
\begin{align*}
\delta_{1} & = \text{dist} \left(E_1,\{x\in\Omega\setminus U: u(x)<1-\ep\} \right) \\
\delta_{2} & = \text{dist} \left(F_1,\{x\in\Omega\setminus U: u(x)>\ep\} \right)
\end{align*}
and let 
\[\delta=\min\left\{\delta_{1},\delta_{2},\text{dist}(E,F)/10,\text{dist}(E,X\setminus\Omega)/2,\text{dist}(F,X\setminus\Omega)/2\right\}.\] 
Since $E_1$ and $F_1$ are 
compact and $u\vert_{\Omega\setminus U}$ is continuous, it follows that $\delta>0$. Let
\begin{align*}
E_{\ep} & = \overline{\bigcup_{x\in E}B(x,\delta)} \\
F_{\ep} & = \overline{\bigcup_{x\in F}B(x,\delta)}
\end{align*}
and let $\eta$ be an $L$-Lipschitz function such that $-1\leq \eta\leq 1$, $\eta=-1$ in $F_{\ep}$, and $\eta=1$ 
in $E_{\ep}$. Let
\[\widetilde{w} = \max \left\{ 0, \frac{u+\eta v-\ep}{1-2\ep} \right\} \]
and let $w=\min\{1,\widetilde{w} \}$. Now $w\in \Nloco(\Omega)$ with $w=0$ in $F_{\ep}$ and $w=1$ in $E_{\ep}$. 
Since $w$ has an upper gradient
\[ g_{w} \le \frac{g_{u}}{1-2\ep} + \frac{Lv+g_{v}}{1-2\ep}, \]
we obtain
\begin{align*}
\newtcapo(E_{\ep},F_{\ep},\Omega) & \leq \int_{\Omega} g_{w} d\mu \\
& = \frac{1}{1-2\ep}\left( \int_{\Omega} g_{u} d\mu + L\int_{\Omega}v\,d\mu + \int_{\Omega}g_{v}d\mu \right) \\
& < \frac{1}{1-2\ep}\newtcapo(E,F,\Omega)+\frac{2\ep}{1-2\ep} \qedhere
\end{align*}
\end{proof}
We recall the following result \cite[Theorem 5.9]{BjoBS08}. The proof given there is valid for 
all $1\leq p<\infty$.
\begin{theorem}[{\cite[Theorem 5.9]{BjoBS08}}]\label{lipdense}
Let $X$ be proper and assume that locally Lipschitz functions are dense in $\Np(X)$. If ~$\Omega\subset X$ is open, 
$u\in\Nlocp(\Omega)$, and $\ep>0$, then
there exists a locally Lipschitz function $v:\Omega\to\R$ such that $\Vert u-v\Vert_{\Np(\Omega)}<\ep$.
\end{theorem}

An analog of the next lemma for $p>1$ was proved in \cite{KalS01}.
\begin{lemma}\label{soblipineq}
Let $\ep$, $\delta$, $E_{\ep}$, $F_{\ep}$, $E$, $F$ and $\Omega$ be as in Lemma~\ref{sobcapineq}. Then
\[
\newtcapo(E_{\ep},F_{\ep}, \Omega) \geq \loclipcapo(E,F,\Omega).
\] 
\end{lemma}

\begin{proof}
Let $\eta_{1}$ be a Lipschitz function such that 
$0\leq \eta_{1}\leq 1$, $\eta_{1}=1$ in $E$ and $\eta_{1}=0$ in $\Omega\setminus E_{\ep}$. Correspondingly, let
$\eta_{2}$ be a Lipschitz function such that $0\leq \eta_{2}\leq 1$, $\eta_{2}=1$ in $F$ and $\eta_{2}=0$ in 
$\Omega\setminus F_{\ep}$. Let $\gamma>0$ and let $u\in\Nloco(\Omega)$ be such that $0\leq u\leq 1$, $u=1$ in 
$E_{\ep}$ and $u=0$ in $F_{\ep}$ with an upper gradient $g_{u}$ such that
\[\int_{\Omega} g_{u} d\mu < \newtcapo(E_{\ep},F_{\ep}, \Omega)+\gamma.\]
By Theorem \ref{lipdense} there is a sequence of locally Lipschitz functions $u_{i}$, $0\leq u_{i}\leq 1$ such that
\[ \lim_{i\to\infty}\Vert u_i-u\Vert_{\No(\Omega)} = 0. \] 
%
Let
\[v_{i} = (1-\eta_{1}-\eta_{2})u_{i}+\eta_{1} =(1-u_i)\eta_1+u_i(1-\eta_2), \]
for $i=1,2,\ldots$ Observe that $v_{i}=1$ in $E$,  $v_{i}=0$ in $F$ and $v_{i}$ has an upper gradient 
\begin{align*} 
g_{v_{i}} & \le g_{\eta_{1}}(1-u_{i}) + \eta_{1}g_{u_{i}} + g_{u_{i}}(1-\eta_{2}) + u_{i}g_{\eta_{2}} \\
& \leq g_{\eta_{1}}(1-u_{i}) + \eta_{1}(g_{u_i-u}+g_{u}) + (g_{u_i-u}+g_{u})(1-\eta_{2}) + u_{i}g_{\eta_{2}}.
\end{align*}
Since we may assume that $g_{u}=0$ in $E_{\ep}\cup F_\ep$, 
$g_{\eta_{1}}=0$ in $\Omega\setminus E_{\ep}$, $g_{\eta_{2}}=0$ in $\Omega\setminus F_{\ep}$ 
and since $u=1$ in $E_{\ep}$ and $u=0$ in $F_{\ep}$ we have
the following estimate
\begin{align*}
\loclipcapo(E,F,\Omega) \leq & \int_{\Omega}g_{v_{i}}d\mu \\
 \leq & \;\Vert g_{\eta_{1}}\Vert_{L^{\infty}(\Omega)}\int_{E_{\ep}}|u_i-u|d\mu + \int_{E_{\ep}}g_{u_i-u}d\mu \\
& + \int_{E_{\ep}}g_{u}d\mu + \int_{\Omega}g_{u_i-u}d\mu + \int_{\Omega} g_{u}d\mu \\
& + \Vert g_{\eta_{2}}\Vert_{L^{\infty}(\Omega)}\int_{F_{\ep}}|u_{i}-u|d\mu.
\end{align*}
%
We may assume that $g_u=0$ in $E_\ep$ because $u=1$ in $E_\ep$. Thus by letting $i\to\infty$ we obtain
\begin{align*}
\loclipcapo(E,F,\Omega) & \leq\int_{\Omega} g_{u} d\mu \\
& \leq \newtcapo(E_{\ep},F_{\ep}, \Omega) + \gamma.
\end{align*}
The claim now follows by letting $\gamma\to 0$.
\end{proof}
Finally we obtain the main result of this section.
\begin{theorem}
If $E$ and $F$ are nonempty disjoint compact subsets of $\Omega$, then
\[\modo(E,F,\Omega) = \capo(E,F,\Omega)=\loclipcapo(E,F,\Omega).\]
\end{theorem}

\begin{proof}
It suffices to prove that
\[ \loclipcapo(E,F,\Omega) \leq \capo(E,F,\Omega) . \]
Let $E_{\ep}$ and $F_{\ep}$ be as in Lemma~\ref{sobcapineq}. Then by Lemma~\ref{soblipineq} and 
Lemma~\ref{sobcapineq},
\begin{align*}
\loclipcapo(E,F,\Omega) & \leq \newtcapo(E_{\ep},F_{\ep}, \Omega) \\
& \leq \frac{1}{1-2\ep}\left(\newtcapo(E,F,\Omega) +2 \ep \right) .
\end{align*}
According to Theorem \ref{onecapsobocap} we have that 
\[\newtcapo(E,F,\Omega)=\capo(E,F,\Omega)\] 
and the claim now follows by letting $\ep\to 0$.
\end{proof} 
%
%
%
%
%
%
\section{Variational BV-capacity}
In this section we study different types of variational BV-capacities in metric measure spaces. Since the class $\BV(X)$ is less restrictive than $N^{1,1}(X)$ in terms of pointwise behavior of functions, it is not obvious which 
definitions of BV-capacity are equivalent. Our main focus in this section is to study which 
pointwise conditions the test functions of capacity must satisfy
in order to establish a BV-capacity that is comparable to the variational $1$-capacity known in literature.

For the Sobolev $1$-capacity, due to the 
quasicontinuity of Sobolev type functions, one can either insist on the test functions $u$ satisfying $u=1$ in a 
neighborhood of the set whose capacity is being computed, or merely require that the test functions satisfies $u=1$ on the set.
Thus the two seemingly different notions of $1$-capacity are equal; for instance, see the earlier discussion in Section~3. 

In the Euclidean setting Federer and Ziemer~\cite{FedZ72} studied a version of 
BV-capacity of a set $E$ based on minimizing the perimeter
of sets that contain $E$ in the interior, and another version of BV-capacity  based on test functions $u$
that satisfy the
requirement that the level set $\{u\ge 1\}$ should have positive upper density 
$\mathcal{H}^{n-1}$-a.e.~in $E$. They proved that these two quantities and the 
variational $1$-capacity are equal. In \cite{KinKST08} Kinnunen et.\/al.~studied the following variational BV-capacity in a 
metric measure space setting. 
\begin{definition}
Let $K\subset X$ be compact. The variational BV-capacity of $K$ is defined as
\[\capbv(K)=\inf\|Du\|(X),\]
where the infimum is taken over all $u\in \BV(X)$ such that $u=1$ on a neighborhood of $K$, $0\leq u\le 1$, 
and the support of $u$ is a compact subset of $X$.
\end{definition}
It was shown in~\cite{KinKST08} that when $K$ is compact the above BV-capacity is comparable to the variational $1$-capacity 
and to the Hausdorff content of codimension one. In this section, we define two modifications of BV-capacity. First we recall the definitions of the \emph{approximate upper and lower limits}.
Let $u$ be a measurable function. Define
\[u^{\vee}(x)=\inf\left\{t\in\mathbb{R}: \lim\limits_{r\to 0}\frac{\mu\left(B(x,r)\cap\left\{u>t\right\}\right)}{\mu\left(B(x,r)\right)}=0\right\},\]
and
\[u^{\wedge}(x)=\sup\left\{t\in\mathbb{R}: \lim\limits_{r\to 0}\frac{\mu\left(B(x,r)\cap\left\{u<t\right\}\right)}{\mu\left(B(x,r)\right)}=0\right\}.\]
Moreover, we define
\[\overline{u}=\frac{ u^{\vee} + u^{\wedge} } {2} .\]
If $u^\vee(x)=u^\wedge(x)$, we denote by
\[\aplim\limits_{y\to x} u(y)=u^\vee(x)\]
the \emph{approximate limit} of $u$ at point $x$.
%
%
The function $u$ is \emph{approximately continuous} at $x$ if
\[\aplim\limits_{y\to x} u(y)=u(x).\]
The \emph{jump set} of function $u$, in the sense of approximate limits, is defined as 
\[S_{u}=\left\{x\in X: u^{\vee}(x)\not=u^{\wedge}(x)\right\}.\]
The classical result states that $S_u$ is countably $(n-1)$-rectifiable for $u\in\BV(\Rn)$, see for instance \cite[Theorem~5.9.6]{Zie89}. This result has an 
analog in the metric setting, where the countable $(n-1)$-rectifiability is replaced with $S_u=\bigcup_{k\in\N}S(k)$ with
$\mathcal{H}(S(k))<\infty$ for each $k\in\N$; this follows
from~\cite[Proposition~5.2 and Theorem~4.4]{AmbMP04}. 

We consider the following capacities in this section.
%
%
\begin{definition}
Let $\Omega\subset X$ be an open set and $K\subset \Omega$ be a compact set.
We define
\[\capone(K,\Omega)=\inf\int_{\Omega}g_{u}d\mu\]
where the infimum is taken over all weak upper gradients $g_{u}$ of functions $u\in N^{1,1}(\Omega)$ such that $\text{supt}(u)$ is a compact subset of $\Omega$ with $u=1$ on a neighborhood of $K$. 
Moreover, we define
\[\cbvo(K,\Omega)=\inf\Vert Du\Vert(\Omega)\]
where the infimum is taken over all $u\in\BV(\Omega)$ such that $0\leq u\leq 1$, $\text{supt}(u)$ is a compact subset of 
$\Omega$ with $u=1$ on a neighborhood of $K$. In this definition, as with $\capone(\cdot,\Omega)$, we can remove
the condition that $0\le u\le 1$. 
We define
\[\cbv(K,\Omega)=\inf\Vert Du\Vert(\Omega)\]
where the infimum is taken over all $u\in\BV(\Omega)$ such that $\text{supt}(u)$ is a compact subset of $\Omega$ and 
$\overline{u}\geq 1$ on $K$. Finally we define
\[\cbvtr(K,\Omega)=\inf\|Du\|(\Omega)\]
where the infimum is taken over all $u\in\BV(\Omega)$ such that  $0\leq u\leq 1$, $\text{supt}(u)$ is a compact subset of 
$\Omega$ and $\overline{u}\geq 1$ on $K$. 
\end{definition}
\begin{remark}
Note that we do not assume that these functions satisfy the condition 
$u\leq 1$ in the definition of $\cbv(K,\Omega)$, though we can assume that $u\ge 0$. 
\end{remark}
For any compact set $K$ we automatically have the following inequalities:
\[\cbv(K,\Omega)\leq\cbvtr(K,\Omega)\leq\cbvo(K,\Omega)\leq\capone(K,\Omega).\]
The first inequality immediately follows from comparing the test functions. The second inequality follows 
from the observation that if $u=1$ in some open neighborhood of a compact set $K$, then
$u^{\wedge}(x)\ge 1$ for every $x\in K$. The third inequality follows from the fact that $N^{1,1}(\Omega)\subset\BV(\Omega)$.

%
%
%
%
%
%
%
%
\begin{theorem}\label{BV,OvsCap}
For any compact set $K\subset\Omega$ we have
\[\capone(K,\Omega)=\cbvo(K,\Omega).\]
\end{theorem}
\begin{proof}
As pointed out above, $\cbvo(K,\Omega)\leq\capone(K,\Omega)$. In order to prove the opposite inequality, let $\ep>0$ 
and $u\in\BV(\Omega)$ be such that $u=1$ in an open neighborhood $U\Subset\Omega$
of $K$, $\text{supt}(u)$ is a compact subset of $\Omega$, and that
\[\Vert Du\Vert(\Omega)<\cbvo(K,\Omega)+\ep.\]
According to Lemma~\ref{compact}, there exist functions $u_{i}\in\mathrm{Lip}(\Omega)$ such that $u_i\to u$ in $L^1(\Omega)$ and satisfying
\[ \lim_{i\to\infty}\int_{\Omega}g_{u_{i}}d\mu=\Vert Du\Vert(\Omega).\] 
We may assume that $0\leq u_i\leq 1$ and by Lemma~\ref{compact}, 
the functions $u_i$ all have a compact support in $\Omega$. 
Let $U'$ be an open set such that $K\subset U'\Subset U$ and $\eta$ be a Lipschitz function such that
$0\leq\eta\leq 1$, $\eta=1$ in $U'$ and $\eta=0$ in $\Omega\setminus U$. 
We define 
\[v_i=\eta+(1-\eta)u_i=(1-u_i)\eta+u_i\]
for $i=1,2,\ldots$ and note that $v_{i}\in N^{1,1}(\Omega)$ with compact support in $\Omega$. Furthermore, the absolute continuity of $v_i$ on $p$-modulus almost every path yields that for every $i\in\N$ the function 
$v_i$ has an upper gradient $g_{v_{i}}$ satisfying
\[
    g_{v_{i}}\le |(1-u_{i})g_{\eta}+(1-\eta)g_{u_{i}}|=(1-u_{i})g_{\eta}+(1-\eta)g_{u_{i}},
\]
%
see \cite[Lemma 3.1]{KinM03} and \cite[Theorem 2.11]{BjoB10}. Since we may assume that $g_\eta=0$ in $\Omega\setminus U$ we have the following estimate
\begin{align*}
\limsup\limits_{i\to\infty}\int_{\Omega}g_{v_{i}}d\mu
&\leq\limsup_{i\to\infty}\int_{\Omega}(1-u_{i})g_{\eta}d\mu+\limsup_{i\to\infty}\int_{\Omega}(1-\eta)g_{u_{i}}d\mu\\
&\leq\Vert g_{\eta}\Vert_{\infty}\limsup_{i\to\infty}\int_{U}|u-u_{i}|d\mu+\Vert Du\Vert(\Omega)\\
&=\Vert Du\Vert(\Omega) < \cbvo(K,\Omega)+\ep.
\end{align*}
Thus, for every $\ep>0$ we find an admissible $N^{1,1}(\Omega)$--function with a compact support in $\Omega$ such that
\[
    \capone(K,\Omega)\leq\limsup_{i\to\infty}\int_{\Omega}g_{v_i}d\mu\leq\Vert Du\Vert(\Omega)<\cbvo(K,\Omega)+\ep.
\]
The claim now follows by letting $\ep\to 0$.
\end{proof}
Before proceeding any further in the study of comparisons of capacities, we give a counterexample demonstrating that $\cbv(K,\Omega)\not=\cbvo(K,\Omega)$ in general.

In this example we use the fact that if a ball $B\subset E$, where $E\subset\mathbb{R}^2$ is a bounded set with finite perimeter, then $P(B)\leq P(E)$. This type of results for convex sets are part
of the general folklore. For the details of one such result which is sufficient for our purposes, see \cite[Proof of Proposition~3.5 and Appendix]{LamP07}.

\begin{example}
Let $m$ denote the ordinary Lebesgue measure and let $X=(\mathbb{R}^{2},|\cdot|,\mu)$ where $d\mu=w\, dm$ with
\begin{align*}
w(x)=
\begin{cases}1&\quad \text{ when } |x|\leq 1,\\
8 &\quad \text{ when } |x|>1.\\
\end{cases}
\end{align*}
Let $K=\overline{B}(0,1)$ and $\Omega=B(0,10)$. Let $\ep>0$. 
The coarea formula implies that 
%
\begin{align*}
\cbvo(K,\Omega)
&\geq\Vert Du\Vert(\Omega)-\ep \\
&\geq P(\{u>t\},\Omega)-\ep
\end{align*}
for some admissible function $u\in\BV(\Omega)$ and for some $0<t<1$. Furthermore, since $u$ is compactly supported in
$\Omega$, via zero extension we obtain that
\[\cbvo(K,\Omega)\ge P(\{u>t\},X)-\ep.\]
From the fact that $u=1$ in a neighborhood of $\overline{B}(0,1)$ 
it follows that $\{u>t\}\Supset B(0,1+\delta)$ for some $\delta>0$.
 
Since the perimeter measure vanishes outside the measure theoretic boundary, we can apply the result mentioned in the discussion preceding this example also in 
this weighted case for sets $\{u>t\}$ and $B(0,1+\delta)$ to obtain
\begin{align*}P(\{u>t\},X)-\ep&\geq P(B(0,1+\delta),X)-\ep=16\pi+16\pi\delta-\ep.
\end{align*}
Thus letting $\ep\to 0$, we get 
\begin{align*}
\cbvo(K,\Omega)&\geq 16\pi.
\end{align*}
By approximating $2\chi_{K}$ with Lipschitz functions 
\[u_i=2\min\left\{1,\max\left\{0,-2i|x|+2i-1\right\}\right\},\]
see \cite[Example 4.5]{HakKin} for details, it follows that 
\[\cbv(K,\Omega)\leq \Vert D(2\chi_{K})\Vert(\Omega)\leq\liminf\limits_{i\to\infty}\int_{\Omega}|Du_i|d\mu\leq 4\pi.\]
Therefore, in this example we get 
\begin{align*}
\cbv(K,\Omega)<\cbvo(K,\Omega).
\end{align*}
\end{example}

In the light of the above example, our goal now is to show that these capacities are comparable. To do so, by 
the previous observations we need to prove that there is a constant $C>0$ such that
\[\frac{1}{C}\,\cbvo(K,\Omega)\leq\cbvtr(K,\Omega)\leq C\,\cbv(K,\Omega)\]
for every compact set $K\subset\Omega$. We start with the following theorem. In the proof we combine discrete convolution and boxing inequality type arguments
to obtain the desired result.
%
%
%
%
%
%
\begin{theorem}
There exists a constant $C>0$, depending only on the constants in the $(1,1)$-Poincaré inequality and the doubling constant of the measure, such that
\[\cbvo(K,\Omega)\leq C\,\cbvtr(K,\Omega)\]
for any compact set $K\subset\Omega$.
\end{theorem}
\begin{proof}
Let $u\in \BV(\Omega)$, with $0\leq u\leq 1$, be an admissible function for computing $\cbvtr(K,\Omega)$, such that
\[\Vert Du\Vert(\Omega)<\cbvtr(K,\Omega)+\ep.\]
Since $\overline{u}\geq 1$ on $K$ and $0\le u\le 1$, we have $u^{\wedge}\geq 1$ on $K$. So by the definition of 
$u^\wedge$, for every $x\in K$ and $t<1$ we have
\[\lim\limits_{r\to 0}\frac{\mu\left(B\left(x,r\right)\cap\left\{u<t \right\}\right)}{\mu\left(B(x,r)\right)}=0.\]
For $t\ge 0$ set $E_t=\{u>t\}$. Thus for every $x\in K$ and $0\le t<1$,
\begin{equation}\label{A-1}
   \lim\limits_{r\to 0}\frac{\mu\left(B\left(x,r\right)\cap E_t\right)}{\mu\left(B(x,r)\right)}=1.
\end{equation}
By the coarea formula there exists $0<t_{0}<1$ such that
\begin{equation}\label{A-2}
   P(E_{t_{0}},\Omega)\le \Vert Du\Vert(\Omega)<\cbvtr(K,\Omega)+\ep.
\end{equation}
%
%
%
%
Let 
\[E=\left\{x\in\Omega:\lim\limits_{r\to 0}\frac{\mu\left(B\left(x,r\right)\cap E_{t_{0}}\right)}{\mu\left(B(x,r)\right)}=1\right\}.\]
Observe that 
$\text{dist}\,(E,X\setminus\Omega)=\delta>0$, and due to the Lebesgue differentiation theorem, 
$\mu(E\Delta E_{t_{0}})=0$. Thus $P(E_{t_{0}},\Omega)=P(E,\Omega)=P(E,X)$. 
The equation~(\ref{A-1}) yields that
%
%
$K\subset E$.
We now apply discrete convolution, as in the proof of Theorem 6.4 in \cite{KinKST08},
to the function $w=\chi_{E}$ with parameter $\rho=\delta/(40\tau)$, where $\tau$ is the constant in the weak $(1,1)$-Poincar\'e 
inequality, to obtain 
\[v(x)=\sum\limits_{i=1}^{\infty}\varphi_i(x)w_{B_{i}}.\]
Here $\left\{B_i\right\}_{i=1}^{\infty}$ is the 
covering related to the discrete convolution. 

Let us make few comments on this technique. The radius of each ball $B_i$ is $\rho=\delta/(40\tau)$. 
The partition of unity $\{\varphi_i\}_{i=1}^{\infty}$ can be constructed in such way that $\varphi_i$ is a $C/\delta$-Lipschitz function with $\varphi_i\geq 1/C$ in $B_i$ and $\text{supt}(\varphi_i)\subset 2B_i$ for every $i=1,2,\ldots$, where the constant $C$  depends only on the overlap constant of the covering related to the discrete convolution, see \cite{KinKST08}. 
The overlap constant does not depend on the radii of the balls used in the covering. We also point out that 
in \cite{KinKST08} the authors assume that $\mu(X)=\infty$, but the construction and properties of discrete convolution do not depend on this assumption at all, and hence that assumption is not needed in this paper.

The proof of Theorem 6.4 in~\cite{KinKST08} implies that $v\in N^{1,1}(\Omega)\cap C(\Omega)$ with 
$0\leq v\leq 1$ and $\text{supt}(v)$ is a compact subset of $\Omega$. 

We divide the set $E$ into two parts
\[E_1=\left\{x\in E:\,\frac{\mu\left(B_i\cap E\right)}{\mu(B_i)}>\frac{1}{2}\; \text{for some}\; B_i\;\text{with}\; x\in B_i\right\}\]
and
\[E_2=\left\{x\in E:\,\frac{\mu\left(B_i\cap E\right)}{\mu(B_i)} \leq\frac{1}{2}\; \text{for every}\;B_i\;\text{with}\;x\in B_i\right\}.\]
%
%
%
%
%
%
Let $x\in E_1$. By the definition of $E_1$ there is a ball $B_i$ such that $x\in B_i$ and $w_{B_i}>1/2$. 
Since $\varphi_i\ge 1/C$ in $B_i$ we have that 
\[
   v\ge \varphi_i w_{B_i}\ge 1/(2C)>0
\]
in $B_i$ and therefore
$\widetilde{v}=2Cv\geq 1$ in a neighborhood of the set $E_1$. Note that $\widetilde{v}\in N^{1,1}(\Omega)\subset \BV(\Omega)$ 
with compact support in $\Omega$ and with the total variation
\[\Vert D\widetilde{v}\Vert(\Omega)\leq C\Vert D\chi_{E}\Vert(\Omega)=C\,P(E,\Omega),\]
see the proof of~\cite[Theorem~6.4]{KinKST08} for details. 

Let $x\in E_2$ and $B_i$ be such that $x\in B_i$. Let $B=B(x,\rho)$ and notice that $B_i\subset 2B\subset 4B_i$. 
Thus we can estimate
\begin{align*}
\frac{\mu\left(2B\cap E\right)}{\mu\left(2B\right)}&=\frac{\mu\left(2B\right)}{\mu\left(2B\right)}-\frac{\mu\left(2B\setminus E\right)}{\mu\left(2B\right)}\\
&\leq 1 - \frac{\mu\left(B_i\setminus E\right)}{\mu\left(4B_i\right)}\\
&\leq 1 - \frac{\mu\left(B_i\setminus E\right)}{C_D^{2}\mu\left(B_i\right)}.
\end{align*}
Here $C_{D}$ is the doubling constant of the measure $\mu$. By the definition of $E_2$ we know that
\[
  \frac{\mu(B_i\setminus E)}{\mu(B_i)}\ge \frac{1}{2}.
\]
Hence 
\[
   \frac{\mu\left(2B\cap E\right)}{\mu\left(2B\right)} \le 1- \frac{1}{2\, C_D^2}=c<1.
\]
%
%
Since $x\in E_2$ is a point of Lebesgue density $1$ for $E_{t_{0}}$ and 
$\mu(E\Delta E_{t_{0}})=0$, we have that
\[\lim\limits_{r\to 0}\frac{\mu\left(B\left(x,r\right)\cap E\right)}{\mu\left(B(x,r)\right)}=1.\]
We now follow the proof of the boxing inequality~\cite[Theorem~3.1]{KinKST08}, see also~\cite[Lemma~3.1]{Mak09}. For every $x\in E_2$ there 
exists $r_x$ with $0<r_x\leq 2\rho$\, such that
\[\frac{\mu\left(B\left(x,r_x\right)\cap E\right)}{\mu\left(B(x,r_x)\right)}\leq c\]
and
\[\frac{\mu\left(B\left(x,r_x/2\right)\cap E\right)}{\mu\left(B(x,r_x/2)\right)}>c.\]

From the choice of $r_x$ it follows that
\begin{align*}\mu\left(B(x,r_x)\setminus E\right)&=\mu\left(B(x,r_x)\right)-\mu\left(B(x,r_x)\cap E\right)\\
&\geq\mu\left(B(x,r_x)\right)-c\, \mu\left(B(x,r_x)\right)\\
&=(1-c)\mu\left(B(x,r_x)\right).
\end{align*}
Furthermore, we have that
\begin{align*}
\frac{\mu\left(B\left(x,r_x\right)\cap E\right)}{\mu\left(B(x,r_x)\right)}&\geq\frac{\mu\left(B\left(x,r_x/2\right)\cap E\right)}{C_{D}\mu\left(B(x,r_x/2)\right)}>\frac{c}{C_{D}}.
\end{align*}
Thus by the relative isoperimetric inequality we obtain
\begin{align*}
\frac{\mu\left(B(x,r_x)\right)}{r_x}&\leq C\, \frac{\min\left\{\mu\left(B\left(x,r_x\right)\cap E\right),\mu\left(B(x,r_x)\setminus E\right)\right\}}{r_x}\\
&\leq C P\left(E,B\left(x,\tau r_x\right)\right),
\end{align*}
where the constant $C$ depends only on the doubling constant of the measure and the constants in the weak $(1,1)$-Poincaré inequality. We apply the standard covering argument to
the family of balls $B(x,\tau r_x)$, $x\in E_2$, to obtain pairwise disjoint balls $B(x_j, \tau r_j)$, $j\in\N$, such that
\[E_2\subset\bigcup\limits_{x\in E_2}B(x,\tau r_x)\subset\bigcup\limits_{j=1}^{\infty}B(x_j,5\tau r_j).\]
By using the doubling property of $\mu$ and the fact that $P(E,\cdot)$ is a Borel measure and that the balls
$B(x_j,\tau r_j)$ are pairwise disjoint, we obtain
\begin{align*}
\sum\limits_{j=1}^{\infty}\frac{\mu\left(B(x_j,5\tau r_j)\right)}{5\tau r_j}
&\le C \sum\limits_{j=1}^{\infty}\frac{\mu\left(B(x_j,r_j)\right)}{r_j}\\
&\leq C\sum\limits_{j=1}^{\infty}P\left(E,B\left(x_j,\tau r_j\right)\right)\\
&=C P\left(E,\bigcup\limits_{j=1}^{\infty}B(x_j,\tau r_j)\right)\\ 
&\leq C P(E,\Omega).
\end{align*}
For positive integers $j$ we define functions $\psi_j\in\text{Lip}(\Omega)$ by
\[\psi_j(x)=\left(1-\frac{\text{dist}\left(x,B(x_j,5\tau r_j)\right)}{5\tau r_j}\right)_{+}.\]
Let
\[\psi=\sup\limits_{1\leq j<\infty}\psi_j\]
and observe that $\psi=1$ in a neighborhood of the set $E_2$. The radius of the covering balls $B(x_j,5\tau r_j)$ satisfies
$5\tau r_j\le\delta/4$ and since $\delta=\text{dist}(E,X\setminus\Omega)\le \text{dist}(E_2,X\setminus\Omega)$, 
we know that $\psi$ has
compact support in $\Omega$. Moreover, $\psi$ has an upper gradient 
\[ g_{\psi}=\sup\limits_{1\leq j<\infty}\frac{1}{5\tau r_j}\chi_{B(x_j,10\tau r_j)},\]
see \cite[Lemma 1.28]{BjoB10}. For the upper gradient we have the estimate
\begin{align*}
\int_{\Omega}g_{\psi}d\mu&\leq\int_{\Omega} \sum\limits_{j=1}^{\infty}\frac{1}{5\tau r_j}\chi_{B(x_j,10\tau r_j)}d\mu\\
&=\sum\limits_{j=1}^{\infty}\int_{B(x_j,10\tau r_j)}\frac{1}{5\tau r_j}d\mu\\
&\leq C_{D}\sum\limits_{j=1}^{\infty}\frac{\mu(B(x_j,5\tau r_j))}{5\tau r_j}\\
&\leq C\,P(E,\Omega).
\end{align*}
Thus $\psi\in N^{1,1}(\Omega)$.
%

Let $\nu=\min\{\widetilde{v} + \psi ,1 \}$ where $\widetilde{v}$ was constructed in considering $E_1$,
and notice that $\nu\in N^{1,1}(\Omega)\subset\BV(\Omega)$, $0\leq\nu\leq 1$ and $\text{supt}(\nu)$ is a 
compact subset of $\Omega$. Moreover $\nu=1$ in a neighborhood of $E\supset K$. By using the properties of total variation measure we have that
\begin{align*}
\capone(K,\Omega)&\leq\|D\nu\|(\Omega)\\
&\leq \|D\widetilde{v}\|(\Omega)+\|D\psi\|(\Omega)\\
&\leq \|D\widetilde{v}\|(\Omega)+\int_{\Omega}g_{\psi}d\mu\\
&\leq C P(E, \Omega).
\end{align*} 
In the previous estimate the third inequality follows from the fact that $\Vert D\psi\Vert(\Omega)\leq\Vert g_\psi\Vert_{L^{1}(\Omega)}$ for
$\psi\in\No(\Omega)$, see \cite[Theorem 6.2.2]{Cam08}.
Since $P(E,\Omega)=P(E_{t_{0}},\Omega)$, the inequality~(\ref{A-2}) yields 
\[\capone(K,\Omega)=\cbvo(K,\Omega)<C\left(\cbvtr(K,\Omega)+\ep\right).\]
The claim now follows by letting $\ep\to 0$.
\end{proof}

Thus $\cbvo(\cdot, \Omega)$ and $\cbvtr(\cdot,\Omega)$ are comparable with a constant which does not depend on $\Omega$. We obtain the remaining inequality in the next theorem.
%
%
%

\begin{theorem}
There exists a constant $C>0$, depending only on the constants in the $(1,1)$-Poincaré inequality and the doubling constant of the measure, such that
\[\cbvtr(K,\Omega)\leq C\,\cbv(K,\Omega)\]
for every compact set $K\subset\Omega$.
\end{theorem}

\begin{proof} 
Let $\ep>0$ and $u\in\BV(\Omega)$ such that $\text{supt}(u)$ is a compact subset of $\Omega$, $\overline{u}\geq 1$ 
in $K$, and
\[\Vert Du\Vert(\Omega)<\cbv(K,\Omega)+\ep.\]
We may assume that $0\leq u\leq 2$ since the truncated function $\widetilde{u}=\min\{u,2\}$ is an admissible function with smaller total variation.
This is due to the observations that 
for any $t<2$ we have $\{u>t\}=\{\widetilde{u}>t\}$ and $\{u<t\}=\{\widetilde{u}<t\}$.
Hence $u^{\vee}(x)\geq 2$ implies that $\widetilde{u}^{\vee}(x)=2$,  and if
$u^{\vee}(x)< 2$ then $u^{\vee}(x)=\widetilde{u}^{\vee}(x)$ and $u^{\wedge}(x)=\widetilde{u}^{\wedge}(x)$.

We divide the set $K$ into two parts;
\[K_{1}=\left\{x\in K: u^{\vee}(x)-u^{\wedge}(x)<1/2\right\}\]
and
\[K_{2}=\left\{x\in K: u^{\vee}(x)-u^{\wedge}(x)\geq 1/2\right\}.\]
Note that $K_{2}\subset S_{u}$ and thus by~\cite[Theorem~5.3]{AmbMP04},
\begin{align}\label{A-4}
\haus(K_{2})\leq \int_{K_2}d\haus
&\leq \frac{2}{\alpha}\int_{K_2}\theta_u d\haus\notag \\
&\leq C\Vert Du\Vert(S_u)\notag \\
&\leq C\Vert Du\Vert(\Omega).
\end{align}
Here $C=2/\alpha$, and we used the estimate that for $x\in K_2$
\[\theta_u(x)\geq\int\limits_{u^\wedge(x)}^{u^\vee(x)}\alpha\,dt\geq\frac{\alpha}{2},\]
where  $\theta_u$ is defined as in \cite[Theorem 5.3]{AmbMP04} and $\alpha>0$ depends only on the doubling constant of the measure and the constants related to  
the weak $(1,1)$-Poincar\'e inequality, see \cite[Theorem 4.4]{AmbMP04} and \cite[Theorem 5.3]{AmbMP04}.
Hence the constant $C$ in the above estimate depends only on the doubling constant of the measure and the constants related to the Poincaré inequality. 
Thus for $\delta>0$ there is a covering
\[\bigcup\limits_{i=1}^{\infty}B(x_{i},r_i)\supset K_{2}\]
such that $r_{i}\leq \delta$ for every $i=1,2,\ldots$ with the estimate
\begin{equation}\label{A-3}
   \sum\limits_{i=1}^{\infty}\frac{\mu\left(B\left(x_i,r_i\right)\right)}{r_i}<\haus(K_{2})+\ep.
\end{equation}
By choosing $\delta$ to be small, we may assume that 
\[\delta<\frac{\text{dist}\left(K,X\setminus\Omega\right)}{10}.\]
We construct an admissible test function to estimate $\cbvtr(K,\Omega)$ as follows. Let us define 
functions $\varphi_i$ by
\[\varphi_i(x)=\left(1-\frac{\text{dist}\left(x,B\left(x_i,r_i\right)\right)}{r_i}\right)_{+}\]
for $i=1,2,\ldots$ and
\[\varphi_{0}=\min\left\{2u,1\right\}.\]
Let
\[\varphi=\sup\limits_{0\leq i<\infty}\varphi_{i}\]
and observe that $0\leq \varphi\leq 1$ and $\text{supt}(\varphi)$ is a compact subset of $\Omega$. If $x\in K_2$, then $x\in B(x_i,r_i)$ 
for some index $1\leq i<\infty$ and $\varphi\geq\varphi_i\geq 1$ in this neighborhood of $x$. 
Thus $\varphi^{\vee}(x)\geq \varphi^{\wedge}(x)\geq 1$ and so $\overline{\varphi}(x)\geq 1$. To obtain similar estimate for the
points of set $K_1$ notice that for $t<1$ we have
\[\left\{\varphi < t\right\}\subset\left\{\varphi_0 < t\right\}\subset\left\{ 2u < t\right\}.\]
Since $u^{\wedge}\geq 1/2$ in $K_1$, it follows that for every $x\in K_1$ and for any $t<1/2$,
\[\lim\limits_{r\to 0}\frac{\mu\left(B(x,r)\cap\left\{u<t\right\}\right)}{\mu\left(B(x,r)\right)}=0.\]
Thus for every $x\in K_1$ and for any $t<1$ we have that
\[\lim\limits_{r\to 0}\frac{\mu\left(B(x,r)\cap\left\{2u<t\right\}\right)}{\mu\left(B(x,r)\right)}=0,\]
and hence
\[\lim\limits_{r\to 0}\frac{\mu\left(B(x,r)\cap\left\{\varphi<t\right\}\right)}{\mu\left(B(x,r)\right)}=0.\]
This implies that $\varphi^{\wedge}\geq 1$ in $K_1$ and therefore $\overline{\varphi}\geq 1$ in 
$K_1$. We have thus obtained that  $\overline{\varphi}\geq 1$ in $K$.
Let
\[\psi_k=\max\limits_{0\leq i\leq k}\varphi_i\]  
and note that $\psi_k\to \varphi$ pointwise as $k\to\infty$. We apply the dominated convergence theorem to obtain 
\[\lim\limits_{k\to\infty}\int_{\Omega}|\varphi-\psi_k|d\mu=\int_{\Omega}\lim\limits_{k\to\infty}\left(\varphi-\psi_k\right)d\mu=0\]
and thus $\psi_k\to \varphi$ in $L^1(\Omega)$ as $k\to\infty$. By the properties of the total variation measure we have
\begin{align*}\Vert D\varphi\Vert(\Omega)&\leq\liminf\limits_{k\to\infty}\Vert D\psi_k\Vert(\Omega)\\
&\leq\liminf\limits_{k\to\infty}\sum\limits_{i=0}^{k}\Vert D\varphi_i\Vert(\Omega)\\
&=\Vert D\varphi_0\Vert(\Omega)+\sum\limits_{i=1}^{\infty}\Vert D\varphi_i\Vert(\Omega).
\end{align*}
Using the definitions of functions $\varphi_i$ we have the estimates
\[\Vert D\varphi_0\Vert (\Omega)\leq 2\Vert Du\Vert(\Omega),\]
and by~(\ref{A-3}) and~(\ref{A-4}),
\begin{align*}
\sum\limits_{i=1}^{\infty}\Vert D\varphi_i\Vert(\Omega)
&\leq\sum\limits_{i=1}^{\infty}\int_{B(x_i,2r_i)}\frac{1}{r_i}d\mu\\
&\leq C_{D}\sum\limits_{i=1}^{\infty}\frac{\mu\left(B\left(x_{i},r_i\right)\right)}{r_i}\\
&\leq C_{D}\left(C\Vert Du\Vert(\Omega)+\ep\right)\\
&\leq C\left(\Vert Du\Vert(\Omega)+\ep\right).
\end{align*}
Thus
\begin{align*}
\cbvtr(K,\Omega)&\leq\Vert D\varphi\Vert(\Omega)\\
&\leq C\left(\Vert Du\Vert(\Omega)+\ep\right)\\
&\leq C\left(\cbv(K,\Omega)+\ep\right).
\end{align*}
The claim now follows by letting $\ep\to 0$. \qedhere
\end{proof}

The constant $C$ in the previous theorem does not depend on $\Omega$. We have thus obtained that there exists a constant $C>0$ such that
\[\frac{1}{C}\,\cbvo(K,\Omega)\leq\cbvtr(K,\Omega)\leq C\,\cbv(K,\Omega)\]
for any compact set $K\subset\Omega$. Combining these estimates we may state the following.

%
%
%
%
%
\begin{corollary}
Let $\Omega\subset X$ be open and $K$ a compact subset of $\Omega$. Then
\[
\capone(K,\Omega)=\cbvo(K,\Omega)\approx\cbvtr(K,\Omega)\approx\cbv(K,\Omega),
\]
with the constants of comparison depending only on the doubling constant of the measure and the constants in the Poincaré inequality.
\end{corollary}
Thus, on the collection of all compact subsets of $\Omega$
we have $\capone(\cdot,\Omega)$, $\cbv(\cdot,\Omega)$, 
$\cbvtr(\cdot,\Omega)$ and $\cbvo(\cdot,\Omega)$ are equivalent by two sided estimates.
%
%
%
%
%
%
%
%
%
\section{Codimension one Hausdorff measure and Capacity}
In this section we study comparisons between the variational capacity and the capacity where the norm of the function is also included.
In the previous sections we studied relative capacities of compact sets. Should the metric space $X$ be parabolic,
then when $\Omega=X$ the relative capacities of every set are zero; 
see for example the discussion in~\cite[Section~7]{KinKST08}. However, the total capacities considered in this section do not have
this drawback. Furthermore, since capacities in general are not inner measures, it would be insufficient for potential theory
to consider only compact sets; in this section we consider more general sets.
 
Let $B\subset X$ be a ball. For $E\subset B$, we define the following four versions of capacity:
\begin{align*}
  \text{cap}_{1}(E, B)&=\inf\int_B g_u\, d\mu,\\
  \text{cap}_{\BV,\text{O}}(E, B)&=\inf\Vert Du\Vert(B),\\
  \text{Cap}_{1}(E, B)&=\inf\int_B\left(|u|+g_u\right)\, d\mu,\\
  \text{Cap}_{\BV,\text{O}}(E, B)&=\inf\left(\int_B |u|\, d\mu+\Vert Du\Vert(B)\right),
\end{align*}
where the first and third infimum are taken over all $u\in N^{1,1}(B)$ and all upper gradients $g_u$ of $u$ such that $u=1$ on $E$
and $u$ has compact support in $B$, 
and the second and fourth infimum are taken over all
$u\in\BV(B)$ such that $u=1$ in a neighborhood of $E$ and $u$ has compact support in $B$. In each of the above definitions we consider the infimum to be infinite, if there are no such functions. 
These definitions are consistent with the
notions studied in previous sections for the case that $E$ is compact.
Observe that when $E\Subset B$ we could also require the test functions
to satisfy the condition $u=1$ in a neighborhood of $E$ in the definition of $\text{cap}_{1}(E,B)$ and
$\text{Cap}_{1}(E,B)$ to obtain the same quantities. This fact can be proved by using a cutoff function and quasicontinuity of functions in $N^{1,1}(B)$, see for instance~\cite[Theorem~6.19]{BjoB10} for
the proof of an analogous result.

\begin{remark}
We could actually extend the above definitions of capacities to arbitrary sets $E\subset X$ by requiring that the test functions satisfy the pointwise conditions in $E\cap B$.
\end{remark}

Our goal is to prove that the above capacities have the same null sets. In order to compare variational capacities with the total capacities 
where the norm of the function is also included, we apply a metric measure space version of Sobolev inequality, see~\cite{HajK00},~\cite[Lemma~2.10]{KinS01},~\cite[Proposition~3.1]{Bjo02}
and~\cite[Section~5.4]{BjoB10} for general discussion.

\begin{lemma} 
Let $B\subset X$ be a ball such that $X\setminus \overline{B}$ is nonempty. For every $E\subset B$ we have
\[\mathrm{cap}_{1}(E,B)\approx \mathrm{Cap}_{1}(E,B)\] 
with the constants of comparison depending solely on the doubling constant, the Poincar\'e constants, and the 
ball $B$. 
\end{lemma}

\begin{proof}
To see this, note that clearly $\text{cap}_{1}(E,B)\le\text{Cap}_{1}(E,B)$. 
In order to prove the remaining comparison, we may assume that $\text{cap}_{1}(E,B)<\infty$. Suppose $u$ is in $N^{1,1}(B)$
with $u=1$ on $E$ and $u$ has a compact support in $B$. 
Since $X\setminus \overline{B}$ is nonempty, it contains a ball which has a positive measure according to the assumptions we made at beginning of this article. 
Thus $X\setminus B$ has a positive capacity and we can apply the
``Poincaré inequality for $N_{0}^{1,1}(B)$-functions'', see \cite[Corollary~5.48]{BjoB10},
to obtain
\[
  \int_B|u|\, d\mu 
                              \le C\int_Bg_u\, d\mu.
\]
Here the constant $C$ depends on the ball $B$ as well as the doubling constant of measure and constants in the
weak $(1,1)$-Poincar\'e inequality. 
In the above estimate we used the fact that since $u$ has a compact support in $B$ we can take $g_u=0$ on 
$X\setminus B$. 
Hence
\[
  \text{Cap}_{1}(E,B)\le \int_B(|u|+g_u)\, d\mu\le C\int_B g_u\, d\mu.
\]
Taking the infimum over all such $u$ gives
\[
  \text{Cap}_{1}(E,B)\le C\ \text{cap}_{1}(E,B).\qedhere
\] 
\end{proof}
In fact, since $\Vert Du\Vert(B)$ is defined using approximation by Lipschitz functions, 
we can use same arguments as in the above proof together with Lemma~\ref{compact} to obtain an analogous result for the corresponding BV-capacities. 
In this case we have the inequality
\[\int_B|u|\, d\mu\leq C\Vert Du\Vert(B)\]
and this yields
\[\text{Cap}_{\BV,\text{O}}(E,B)\le C\Vert Du\Vert(B).\]
Thus we have the following.
\begin{corollary} 
Let $B\subset X$ be a ball such that $X\setminus \overline{B}$ is nonempty. For every $E\subset B$ we have
\[
   \mathrm{cap}_{\BV,\mathrm{O}}(E,B)\approx \mathrm{Cap}_{\BV,\mathrm{O}}(E,B),
\] 
with the constants of comparison
depending solely on the doubling constant, the Poincar\'e constants, and the 
ball $B$. 
\end{corollary}
Furthermore, if $K\subset B$ is compact, then according to Theorem \ref{BV,OvsCap} in the previous section
$\text{cap}_{\BV,\text{O}}(K,B)=\text{cap}_{1,\text{O}}(K,B)=\text{cap}_{1}(K,B)$. 

For more general sets $E$,  we consider
\begin{align*}
    \text{Cap}_1(E)&=\inf\int_X(|u|+g_u)\, d\mu,\\
    \text{Cap}_{\BV,\mathrm{O}}(E)&=\inf\left(\int_X u\,d\mu+\Vert Du\Vert(X)\right)
\end{align*}
where the first infimum is taken over
all $u\in N^{1,1}(X)$ such that $u=1$ on $E$ and the second infimum is taken over all $u\in \BV(X)$ such that $u=1$ on an open neighborhood $E$. 
It was shown in~\cite{KinKST08} that the set of non-Lebesgue points of a function in $N^{1,1}(X)$ is of zero $\text{Cap}_1$--capacity.

It can be seen that $\text{Cap}_1(E)=0$ if and only if for all balls $B$ in $X$, 
$\text{Cap}_{1}(E\cap B,2B)=0$. Indeed, let us assume that the latter condition holds. By comparing the test functions we immediately have that 
\[\text{Cap}_{1}(E\cap B)\le \text{Cap}_{1}(E\cap B,2B)\]
for all balls $B\subset X$.
Thus, if a ball $B\subset X$ is fixed, then the countable subadditivity of $1$-capacity, see for instance \cite{BjoB10}, implies that
\[\text{Cap}_{1}(E)\leq\sum\limits_{k=1}^{\infty}\text{Cap}_{1}(E\cap kB)\leq\sum\limits_{k=1}^{\infty}\text{Cap}_{1}(E\cap kB,2kB)=0.\]
In order to prove the opposite inequality, let us assume that $\text{Cap}_1(E)=0$. Let $B\subset X$ be a ball and let $\ep>0$. There is a function
$u\in N^{1,1}(X)$ such that $u=1$ on $E$ and that
\[\int_X(|u|+g_u)\, d\mu<\ep.\]
Let $v=u\eta$ where $\eta\in\text{Lip}(2B)$ is such that $0\leq\eta\leq 1$, $\eta=1$ in $B$ and that $\text{supt}(\eta)$ is a compact subset of $2B$. Denote by $L>0$ the Lipschitz constant of $\eta$. Then $\eta g_{u}+L|u|$ is a weak upper gradient of $v$ and
\begin{align*}\text{Cap}_{1}(E\cap B,2B)&\leq\int_{2B}\left(|v|+\eta g_{u}+L|u|\right)\, d\mu\\
&\leq(L+1)\int_{X}(|u|+g_u)\, d\mu\\
&<(L+1)\ep.
\end{align*}
The desired result then follows by letting $\ep\to 0$.

The previous argument can be used to obtain the corresponding result for BV-capacities. Indeed the BV-capacity $\text{Cap}_{\BV,\mathrm{O}}(\cdot)$ is countable subadditive, see~\cite[Theorem~3.3]{HakKin}, 
and thus $\text{Cap}_{\BV,\mathrm{O}}(E\cap B,2B)=0$ for all balls $B$ in $X$ implies similarly that $\text{Cap}_{\BV,\mathrm{O}}(E)=0$. On the other hand, if $\text{Cap}_{\BV,\mathrm{O}}(E)=0$, then
the equivalence of capacities, see~\cite[Theorem~4.3]{HakKin}, implies that $\text{Cap}_1(E)=0$. Thus the previous statement concerning $1$-capacities together with the fact that $N^{1,1}(2B)\subset \BV(2B)$ gives
\[\text{Cap}_{\BV,\mathrm{O}}(E\cap B,2B)\leq\text{Cap}_{1}(E\cap B,2B)=0\]
for all balls $B\subset X$.

Let us recall that the Theorems 4.3 and 5.1 in~\cite{HakKin} imply that for any set $E\subset X$
\[
\text{Cap}_{\text{BV,O}}(E)=0 \Longleftrightarrow\text{Cap}_{1}(E)=0\Longleftrightarrow\mathcal{H}(E)=0.
\]
The following corollary is a consequence of the results studied above in this section. 
In this result we assume that $X\setminus\overline{2B}$ is nonempty. However, 
if $X\setminus\overline{2B}$ is empty, then $\overline{2B}=X$ and the following versions
of total $1$-capacities coincide, as do the total BV-capacities; however, the variational capacities are all zero.

\begin{corollary}
Let $B\subset X$ be a ball such that $X\setminus \overline{2B}$ is nonempty and let $E\subset B$. Then
the following are equivalent.
\begin{enumerate}
\item $\mathrm{Cap}_1(E)=0$,
\item $\mathrm{Cap}_{\BV,\mathrm{O}}(E)=0$,
\item $\mathrm{Cap}_{1}(E,2B)=0$,
\item $\mathrm{Cap}_{\BV,\mathrm{O}}(E,2B)=0$,
\item $\mathrm{cap}_{1}(E,2B)=0$,
\item $\mathrm{cap}_{\BV,\mathrm{O}}(E,2B)=0$,
\item $\mathcal{H}(E)=0$.
\end{enumerate}
\end{corollary}

In potential theory one is usually concerned with whether a set is of zero capacity or not.
The above corollary shows that any of the six forms of capacity can be used in studying this question.
Furthermore, if $\text{Cap}_{\BV,\mathrm{O}}(E,B)=0$ then $\mathcal{H}(E)=0$.

The above type total capacities and their connections to Hausdorff measure of codimension one were studied in~\cite{HakKin}. The relative
variational BV-capacity in metric measure space setting was studied in~\cite{KinKST10} in the context of De Giorgi measure
and an obstacle problem. In~\cite{KinKST10}, the authors require the capacity test functions for $\text{cap}_{\BV,\text{O}}(E,B)$ to be 
zero in $X\setminus\overline{B}$. As an interesting result~\cite[Corollary~6.4]{KinKST10} they obtain that if $\text{cap}_{\BV,\text{O}}(E,B)>0$, then
\[\mathcal{H}(\partial^{*} G)\approx \text{cap}_{\BV,\text{O}}(E,B),\]
for some $G\in\mathcal{G}$ satisfying $E\subset\text{int}\,G$. The class $\mathcal{G}$, related to De Giorgi measure, is a refined collection of measurable sets
satisfying certain density conditions, see~\cite{KinKST10} for details.




\noindent Address:\\

\noindent H.H:  Department of Mathematical Sciences, 
P.O. Box 3000, 
FI-90014 University of Oulu, Finland. \\
\noindent E-mail: {\tt heikki.hakkarainen@oulu.fi}\\

\noindent N.S.: Department of Mathematical Sciences, P.O.Box 210025, University of
Cincinnati, Cincinnati, OH 45221--0025, U.S.A. \\
\noindent E-mail: {\tt nages@math.uc.edu} \\

\begin{thebibliography}{99}
\bibitem{Amb01}L. Ambrosio, \emph{Some fine properties of sets of finite perimeter in Ahlfors regular metric measure spaces},
 Adv. Math.  \textbf{159}  (2001),  no. 1, 51--67. 
\bibitem{Amb02}L. Ambrosio, \emph{Fine properties of sets of finite perimeter in doubling metric measure spaces},
Calculus of variations, nonsmooth analysis and related topics.
Set-Valued Anal. \textbf{10} (2002), no. 2-3, 111--128.
\bibitem{Amb03}L. Ambrosio, \emph{On some recent developments of the theory of sets of finite perimeter},
 Atti Accad. Naz. Lincei Cl. Sci. Fis. Mat. Natur. Rend. Lincei (9) Mat. Appl.  \textbf{14}  (2003),  no. 3, 179--187.
 \bibitem{AmbM03}L. Ambrosio and V. Magnani, \emph{Weak differentiability of BV functions on stratified groups},
  Math. Z.  \textbf{245}  (2003),  no. 1, 123--153.
\bibitem{AmbMP04}L. Ambrosio, M. Jr. Miranda, and D. Pallara, \emph{Special functions of bounded variation in doubling metric measure spaces. Calculus of variations: topics from the mathematical heritage of E. De Giorgi}, Quad. Mat. \textbf{14}, Dept. Math., Seconda Univ. Napoli, Caserta, (2004), 1--45.
\bibitem{BjoB10}A. Bj\"orn, J. Bj\"orn, \emph{Nonlinear Potential Theory on Metric Spaces}, to appear in EMS Tracts in Mathematics, European Mathematical Society, Zurich.
\bibitem{BjoBS08}A. Bj{\" o}rn, J. Bj{\" o}rn, and N. Shanmugalingam, \emph{Quasicontinuity of Newton-Sobolev functions and density of Lipschitz functions on metric spaces}, Houston J. Math. \textbf{34} (2008), no. 4, 1197--1211.
\bibitem{Bjo02}J. Bj{\" o}rn, \emph{Boundary continuity for quasiminimizers on metric spaces}, Illinois J. Math. \textbf{46} (2002), no. 2, 383--403.
\bibitem{BobH97} S.G. Bobkov, C. Houdré, \emph{Some connections between isoperimetric and Sobolev-type inequalities}, Mem. Amer. Math. Soc. \textbf{129} (1997), no. 616, viii+111 pp.
\bibitem{Cam08}C. Camfield, \emph{Comparison of BV norms in weighted Euclidean spaces and metric measure spaces},
Thesis (Ph.D.)--University of Cincinnati (2008), 141 pp.
\bibitem{Cos09}S. Costea, \emph{Sobolev capacity and Hausdorff measures in metric measure spaces}, Ann. Acad. Sci. Fenn. Math. \textbf{34} (2009), no. 1, 179--194.
\bibitem{EvaG92}L.C. Evans and R.F Gariepy, \emph{Measure theory and fine properties of functions} Studies in Advanced Mathematics series, CRC Press, Boca Raton, 1992.
\bibitem{FedZ72}H. Federer, W.P. Ziemer, \emph{The Lebesgue set of a function whose distribution derivatives are $p$-th power summable} Indiana Univ. Math. J. \textbf{22} (1972/73), 139--158.
\bibitem{GolT02}V. Gol'dshtein, M. Troyanov, \emph{Capacities in metric spaces}, Integral Equations Operator Theory \textbf{44} (2002), no. 2, 212--242.
\bibitem{Gus60}W. Gustin, \emph{Boxing inequalities}, J. Math. Mech. \textbf{9} (1960) 229--239.
\bibitem{Haj03}P. Haj\l asz, \emph{Sobolev spaces on metric-measure spaces}, 
Heat kernels and analysis on manifolds, graphs, and metric spaces (Paris, 2002),  173--218,
Contemp. Math., \textbf{338}, Amer. Math. Soc., Providence, RI, (2003).
\bibitem{HajK00}P. Haj\l asz, P. Koskela, \emph{Sobolev met Poincaré}, Mem. Amer. Math. Soc. \textbf{145} (2000), no. 688, x+101 pp.
\bibitem{HakKin}H. Hakkarainen and J. Kinnunen, \emph{The BV-capacity in metric spaces},  Manuscripta Math.  
 \textbf{132}  (2010),  no. 1-2, 51--73.
\bibitem{Hei01}J. Heinonen, \emph{Lectures on analysis on metric spaces} Universitext. Springer-Verlag, New York, 2001.
\bibitem{HeiK98}J. Heinonen and P. Koskela, \emph{Quasiconformal maps in metric spaces with controlled geometry} Acta Math. \textbf{181} (1998), no. 1, 1--61.
\bibitem{Hut}J. E. Hutchinson, \emph{On the relationship between Hausdorff measure and a measure of 
De Giorgi, Colombini, and Piccinini}, Boll. Un. Mat. Ital. B (5) \textbf{18} (1981), no. 2, 619--628.
\bibitem{JarJRRS07}E. J{\" a}rvenp{\" a}{\" a}, M. J{\" a}rvenp{\" a}{\" a}, K. Rogovin, S. Rogovin, and N. Shanmugalingam, \emph{Measurability of equivalence classes and ${\rm MEC}_p$-property in metric spaces} Rev. Mat. Iberoam. \textbf{23} (2007), no. 3, 811--830.
\bibitem{KalS01}S. Kallunki and N. Shanmugalingam, \emph{Modulus and continuous capacity} Ann. Acad. Sci. Fenn. Math. \textbf{26} (2001), no. 2, 455--464. 
\bibitem{KilKM00}T. Kilpeläinen, J. Kinnunen, O. Martio, \emph{Sobolev spaces with zero boundary values on metric spaces}, Potential Anal. \textbf{12} (2000), no. 3, 233--247.
\bibitem{KinKST08}J. Kinnunen, R. Korte, N. Shanmugalingam, and H. Tuominen,  \emph{Lebesgue points and capacities via the boxing inequality in metric spaces} Indiana Univ. Math. J. \textbf{57} (2008), no. 1, 401--430.
\bibitem{KinKST10}J. Kinnunen, R. Korte, N. Shanmugalingam, and H. Tuominen, \emph{The DeGiorgi measure and an 
obstacle problem related to minimal surfaces in metric spaces}, J. Math. Pures Appl. \textbf{93} (2010), 599--622.
\bibitem{KinM03}J. Kinnunen and O. Martio, \emph{Sobolev space properties of superharmonic functions on metric spaces},
Result Math. \textbf{44} (2003), no. 1-2, 114--129.
\bibitem{KinMa1}J. Kinnunen and O. Martio, \emph{Nonlinear potential theory on metric spaces} Illinois J. Math. \textbf{46}
(2002), no. 3, 857--883.
\bibitem{KinMa2}J. Kinnunen and O. Martio, \emph{Choquet property for the Sobolev capacity in metric spaces}
Proceedings on Analysis and Geometry (Russian),Novosibirsk Akademgorodok (1999), 285--290.
\bibitem{KinMa3}J. Kinnunen and O. Martio, \emph{The Sobolev capacity on metric spaces} Ann. Acad. Sci. Fenn. Math.
\textbf{21} (1996), no. 2, 367--382.
\bibitem{KinS01} J. Kinnunen and N. Shanmugalingam, \emph{Regularity of quasi-minimizers on metric spaces} Manuscripta Math. \textbf{105} (2001), 401--423.
\bibitem{Kor07}R. Korte, \emph{Geometric implications of the Poincar\'e inequality},
 Results Math.  \textbf{50}  (2007),  no. 1-2, 93--107.
\bibitem{LamP07}J. Lamboley, M. Pierre, \emph{Structure of shape derivatives around irregular domains and applications}, J. Convex Anal. \textbf{14} (2007), no. 4, 807--822.
\bibitem{Maz}V. Maz'ya, \emph{Sobolev spaces} Translated from the Russian by T. O. Shaposhnikova. Springer Series in Soviet Mathematics. Springer-Verlag, Berlin (1985), xix+486 pp.
\bibitem{Mir03}M. Miranda Jr., \emph{Functions of bounded variation on "good'' metric spaces},
J. Math. Pures Appl. (9)  \textbf{82}  (2003),  no. 8, 975--1004.
\bibitem{Mak09}T. M{\" a}k{\" a}l{\" a}inen, \emph{Adams inequality on metric measure spaces} Rev. Mat. Iberoam. \textbf{25} (2009), no. 2, 533--558.
\bibitem{Pi73}L.C. Piccinini, \emph{De Giorgi's measure and thin obstacles}, Geometric Measure Theory and Minimal Surfaces, 
Centro Internaz. Mat. Estivo (C.I.M.E.), III Ciclo, Varenna (1972), Edizioni Cremonese, Rome, 1973, pp. 221--230.
\bibitem{Sha00}N. Shanmugalingam, \emph{Newtonian spaces: An extension of {S}obolev spaces to metric measure spaces} Rev. Mat. Iberoamericana 16(2) (2000), 243--279.
\bibitem{Zie89}W.P. Ziemer, \emph{Weakly differentiable functions. Sobolev spaces and functions of bounded variation} Graduate Texts in Mathematics, 120. Springer-Verlag, New York, 1989. 

\end{thebibliography}
\end{document}